\documentclass{amsart}

\usepackage{amsmath}
\usepackage{amscd}

\usepackage{graphicx}
\newtheorem{theorem}{Theorem}[section]
\newtheorem{lemma}[theorem]{Lemma}

\theoremstyle{definition}

\theoremstyle{proposition}
\newtheorem{proposition}[theorem]{Proposition}
\newtheorem{remark}[theorem]{Remark}
\theoremstyle{corollary}

\numberwithin{equation}{section}

\begin{document}

\title{A generalization of Riemann's theta functions for singular curves}
%%% \subtitle{Sub-Title of Your Article}%%% optional

\author{Yukitaka \textsc{Abe}}% Author Name (\sc should NOT be used here)
\address{Department of Mathematics, University of Toyama, Toyama 930-8555, Japan}
\email{abe@sci.u-toyama.ac.jp}

%\author{Daz \textsc{Baz}}% Author Name (\sc should NOT be used here)
%\address{Address of the organization to which the second author beongs}
%\email{mail\_address\_of\_the\_second\_author}

% Subject code(s)

% Key word(s)

%\recdate{2016}{  }{  }% Date of reception
%\revdate{201?}{ }{  }%  Date of revision

\begin{abstract}
Let $X$ be a compact Riemann surface of genus $g$.
Jacobi's inversion theorem states that the Abel-Jacobi map
$\varphi : X^{(g)} \longrightarrow J(X)$ is surjective, where
$X^{(g)}$ is the symmetric product of $X$ of degree $g$
and $J(X)$ is the Jacobi variety of $X$. Riemann obtained
the explicit solution of the Jacobi inversion problem
introducing Riemann's theta functions.
We study such a problem for singular curves. We define
a generalization of Riemann's theta functions and Riemann's constants.
We obtain similar results for singular curves.
\end{abstract}

\maketitle

\section{Introduction}

Let $X$ be a compact Riemann surface of genus $g$. A basis of the space of holomorphic
1-forms on $X$ gives the Abel-Jacobi map $\varphi : X^{(g)} \longrightarrow J(X)$,
where $X^{(g)}$ is the symmetric product of $X$ of degree $g$ and $J(X)$ is the Jacobi
variety of $X$. We know that it is biholomorphic. Moreover, Riemann obtained the
explicit solution of the Jacobi inversion problem introducing Riemann's theta functions.
\par
In this paper we generalize the above argument to singular curves. We constructed
generalized Jacobi varieties of singular curves completely analytically, and established
fundamental properties in \cite{ref2} and \cite{ref5}.
We treat only singular curves with an ordinal double point 
to avoid complicated expression.
However we think that the method is valid for general singular curves.
\par
The paper is organized as follows.
In Section 2 we recall some fundamental properties of singular curves and
their analytic Albanese varieties. We also give a remark on the Remmert-Morimoto
decomposition of commutative complex Lie groups and the meromorphic
function fields on their standard compactifications. 
In Section 3 we state the precise setting in this paper. We define a
generalized theta function in Section 4. We use it through the paper.
In Section 5 we generalize Riemann's constants and solve explicitly the Jacobi inversion
problem for singular curves (Theorem 5.1). In the last section, we show that
the image of a singular curve excluded a small neighbourhood of a singular
point is contained in the zero set of a generalized theta function (Theorem 6.6).

\section{A remark on analytic Albanese varieties}
We shall review here some fundamental properties of singular curves and their analytic
Albanese varieties. We refer to \cite{ref2} for details.\par
Let $X$ be a compact Riemann surface with the structure sheaf ${\mathcal O}_{X}$.
Take a finite subset $S$ of $X$. We consider an equivalence relation $R$ on $S$.
We define the quotient set $\overline{S} := S/R$ of $S$ by $R$.
We set
\[\overline{X} := (X \setminus S)\cup \overline{S}.\]
We induce to $\overline{X}$ the quotient topology
by the canonical projection $\rho : X \longrightarrow \overline{X}$.
Then $\overline{X}$ is a compact Hausdorff space.
According to Serre (\cite{ref6}), we define
a modulus ${\mathfrak m}$ with support $S$ by the data of an
integer ${\mathfrak m}(P)>0$ for each point $P \in S$.
Let $\rho _* {\mathcal O}_X$ be the direct image of
${\mathcal O}_X$ by the projection
$\rho$. For any $Q \in \overline{S}$ we denote by ${\mathcal I}_Q$
the ideal of $(\rho _* {\mathcal O}_X)_Q$ formed by germs of
functions $f$ with ${\rm ord}_P(f) \geq {\mathfrak m}(P)$ for
any $P \in \rho ^{-1}(Q)$.
We define a sheaf ${\mathcal O}_{\mathfrak m}$ on $\overline{X}$ by
\[
{\mathcal O}_{{\mathfrak m} , Q} :=
%\left\{
\begin{cases}
(\rho _*{\mathcal O}_X)_Q = {\mathcal O}_{X, Q}&
\mbox{if $Q \in X \setminus S$},\\
{\mathbb C} + {\mathcal I}_Q&
\mbox{if $Q \in \overline{S}$}.
\end{cases}\]
Then we obtain a 1-dimensional compact reduced complex
space $(\overline{X}, {\mathcal O}_{\mathfrak m})$, which
we denote by $X_{{\mathfrak m}}$.
\par
Conversely, any reduced and irreducible singular curve is obtained
as above.\par
Let $g$ be the genus of $X$. For any $Q \in X_{{\mathfrak m}}$ we set
$\delta _{Q} := \dim \left((\rho _{*}{\mathcal O}_X)_{Q}/{\mathcal O}_{{\mathfrak m}, Q}\right)$.
The genus of $X_{{\mathfrak m}}$ is defined by $\pi := g + \delta$,
where
$\delta := \sum_{Q \in X_{{\mathfrak m} }} \delta _Q $.
We denote by $\Omega _{{\mathfrak m}}$ the duality sheaf on $X_{{\mathfrak m}}$.
We have
\[\dim H^0(X_{{\mathfrak m}}, \Omega _{{\mathfrak m}}) = \dim
H^1(X_{{\mathfrak m}}, {\mathcal O}_{{\mathfrak m}}) = \pi .\] 
\par
We take a basis $\{ \omega _1, \dots , \omega _{\pi }\}$ of
$H^0(X_{{\mathfrak m}}, \Omega _{{\mathfrak m}})$.
Let $\{ \alpha _1, \beta _1, \dots , \alpha _g,
\beta _g \}$ be a canonical homology basis of $X$. 
We set $S = \{ P_1, \dots , P_s\}$. We consider a small
circle $\gamma _j$ centered at $P_j$ with anticlockwise direction for
$j = 1, \dots ,s$. Then the set $\{ \alpha _1, \beta _1, \dots , \alpha _g,
\beta _g, \gamma _1, \dots , \gamma _{s-1}\}$ forms a basis of
$H_1(X \setminus S, {\mathbb Z}) = H_1(X_{{\mathfrak m}}\setminus
\overline{S}, {\mathbb Z})$. Let $H^0(X_{\mathfrak m},\Omega _{\mathfrak m})^{*}$ be the dual space of
$H^0(X_{\mathfrak m},\Omega _{\mathfrak m})$. We set
\[A := H^0(X_{\mathfrak m},\Omega _{\mathfrak m})^{*}/
H_1(X_{\mathfrak m} \setminus \overline{S},{\mathbb Z}).\]
Consider $2g + s -1$ vectors
\[\left(\int _{\alpha _i}\rho ^{*}\omega _1, \dots
, \int _{\alpha _i}\rho ^{*}\omega _{\pi }\right),\quad i = 1, \dots , g,\]
\[\left(\int _{\beta _i}\rho ^{*}\omega _1, \dots
, \int _{\beta _i}\rho ^{*}\omega _{\pi }\right),\quad i = 1, \dots , g\]
and
\[\left(\int _{\gamma _j}\rho ^{*}\omega _1, \dots
, \int _{\gamma _j}\rho ^{*}\omega _{\pi }\right),\quad j = 1, \dots , s-1\]
in ${\mathbb C}^{\pi}$.  Let $\Gamma $ be a subgroup of ${\mathbb C}^{\pi}$
generated by these vectors over ${\mathbb Z}$.
We have $A \cong {\mathbb C}^{\pi }/\Gamma $ as a complex Lie group.
We call $A$ with the structure as a complex Lie group the analytic Albanese
variety of $X_{{\mathfrak m}}$, and write it as
${\rm Alb}^{an}(X_{{\mathfrak m}})$.
We define a period map $\varphi $ with base point
$P_0 \in X \setminus S$ by 
\[\varphi : X \setminus S \longrightarrow
A,\quad
P \longmapsto \left[ \left( \int _{P_0}^{P}\rho ^{*} \omega _1, \dots ,
\int _{P_0}^{P}\rho ^{*} \omega _{\pi}\right)\right].\]
This is the Abel-Jacobi map in the case of singular curves.
The map $\varphi $ is extended to a bimeromorphic map
$\varphi : (X \setminus S)^{(\pi )} \longrightarrow A$
(Theorem 5.19 in \cite{ref2}).\par
We give a remark on the canonical form of connected commutative
complex Lie groups. Let ${\mathbb C}^n/\Gamma $ be a connected
commutative complex Lie group. By the theorem of Remmert-Morimoto
we have the canonical form 
\[{\mathbb C}^n/\Gamma \cong {\mathbb C}^p \times ({\mathbb C}^{*})^q
\times G,\]
where $G$ is a toroidal group of dimension $r$ and $n = p+q+r$.
The complex dimension of the complex linear subspace spanned by 
$\Gamma $ is said to be the complex rank of $\Gamma $. We denote
it by ${\rm rank}_{{\mathbb C}}\Gamma $. Obviously,
${\rm rank}_{{\mathbb C}}\Gamma = q + r$ in the above case. Then
${\mathbb C}^n/\Gamma = {\mathbb C}^p \times ({\mathbb C}^{q+r}/\Gamma )$ and ${\mathbb C}^{q+r}/\Gamma \cong ({\mathbb C}^{*})^q \times G$. \par
We consider $X = {\mathbb C}^n/\Gamma $ with ${\rm rank}_{{\mathbb C}} = n$. We denote by ${\mathbb R}_{\Gamma }^{n+m}$ the real linear subspace
generated by $\Gamma $. An $m$-dimensional complex linear subspace
${\mathbb C}_{\Gamma }^{m} := {\mathbb R}_{\Gamma }^{n+m} \cap \sqrt{-1}
{\mathbb R}_{\Gamma }^{n+m}$ is the maximal complex linear subspace contained in ${\mathbb R}_{\Gamma }^{n+m}$. We can always take a period
matrix $P$ of $X$ as follows:
\begin{equation}
P = 
\begin{pmatrix}
0 & I_m & T \\
I_{n-m} & R_1 & R_2\\
\end{pmatrix}\quad
\text{with $\det ( {\mathfrak Im}(T)) \not= 0$},
\end{equation}
where $I_k$ is the unit matrix of degree $k$ and $R_1$ and $R_2$ are real matrices.

\begin{remark}
$X = {\mathbb C}^n/\Gamma $ is a toroidal group if and only if the real 
$(n-m, 2m)$-matrix $(R_1\quad R_2)$ satisfies the irrationality condition (for example, see Proposition 2.2.2 in \cite{ref3}).
\end{remark}

If the irrationality condition is not satisfied, then we have the Remmert-Morimoto decomposition
\[X \cong ({\mathbb C}^{*})^q \times Y,\]
where $q \geq 1$ and $Y = {\mathbb C}^r/\Lambda $ is a toroidal group.
In this case, we have $n = q + r$ and ${\rm rank}\; \Lambda = r + m$.
Let $\sigma _X : {\mathbb C}^n \longrightarrow X$ be the projection. We set
$B := ({\mathbb C}^{*})^q \times Y = {\mathbb C}^{q+r}/({\mathbb Z} \oplus \Lambda )$. We also have the projection $\sigma _B : {\mathbb C}^n \longrightarrow B$. There exists an isomorphism $\varphi : X \longrightarrow B$. If $\Phi : {\mathbb C}^n \longrightarrow {\mathbb C}^n$ is the linear extension of $\varphi $, then $\Phi ({\mathbb C}_{\Gamma }^m) = {\mathbb C}_{\Lambda }^m$ and the following diagram is commutative:
\[
\begin{CD}
{\mathbb C}^n @>\Phi >> {\mathbb C}^n \\
@V\sigma _X VV           @VV\sigma _B V\\
X @>> \varphi > B
\end{CD}\]
\par
Let $\Lambda _{{\mathbb T}}$ be the discrete subgroup of ${\mathbb C}_{\Gamma }^m$ 
with ${\rm rank}\; \Lambda _{{\mathbb T}} = 2m$ 
generated by column vectors of the $(m,2m)$-matrix $(I_m\quad T)$ in (2.1). 
Then ${\mathbb T} := {\mathbb C}_{\Gamma }^m/\Lambda _{{\mathbb T}}$ is an $m$-dimensional complex torus.
From the projection $\mu : {\mathbb C}^n \longrightarrow {\mathbb C}_{\Gamma }^m$ we obtain a principal $({\mathbb C}^{*})^{n-m}$-bundle $\tau : X \longrightarrow {\mathbb T}$. Replacing fibers $({\mathbb C}^{*})^{n-m}$ with $({\mathbb P}^1)^{n-m}$, we have the associated $({\mathbb P}^1)^{n-m}$-bundle $\overline{\tau } : \overline{X} \longrightarrow {\mathbb T}$. A toroidal group is said to be a quasi-abelian variety if it satisfies generalized Riemann conditions. We refer to \cite{ref3} for the definition of the kind of a quasi-abelian variety and other properties. In the above situation we have the following lemma.

\begin{lemma}
If ${\mathbb T}$ is an abelian variety, then $Y$ is a quasi-abelian variety of kind 0.
\end{lemma}

\begin{proof}
There exists an ample Riemann form ${\mathcal H}_{{\mathbb T}}$ for ${\mathbb T}$. We set ${\mathcal H}_0 := {\mathcal H}_{{\mathbb T}}\circ (\mu \times \mu )$. Then ${\mathcal H} := {\mathcal H}_0\circ (\Phi ^{-1} \times \Phi ^{-1})$ is a hermitian form on ${\mathbb C}^n$. By the definition, ${\mathcal H}$ is positive definite on ${\mathbb C}_{\Lambda }^m$ and the imaginary part ${\mathfrak Im}{\mathcal H}$ of ${\mathcal H}$ is ${\mathbb Z}$-valued on $\Lambda \times \Lambda$. Since ${\mathcal H}_{{\mathbb T}}$ is positive definite, we have ${\rm rank}\left( {\mathfrak Im}{\mathcal H}\right)_{{\mathbb R}_{\Lambda }^{r+m}} = 2m$, where $\left({\mathfrak Im}{\mathcal H}\right)_{{\mathbb R}_{\Lambda }^{r+m}}$ is the restriction of ${\mathfrak Im}{\mathcal H}$ onto ${\mathbb R}_{\Lambda }^{r+m} \times {\mathbb R}_{\Lambda }^{r+m}$. Therefore ${\mathcal H}|_{{\mathbb C}^r \times {\mathbb C}^r}$ is an ample Riemann form of kind 0 for $Y$.
\end{proof}

In the rest of this section, we assume that ${\mathbb T}$ is an abelian variety. Then $Y$ is a quasi-abelian variety of kind 0 by the above lemma. Let $\overline{Y}$ be the standard compactification of $Y$ (see \cite{ref2} or \cite{ref3}). Then $\overline{B} = ({\mathbb P}^1)^q \times \overline{Y}$ is the standard compactification of $B$. Let ${\rm Mer}(\overline{B})$ be the field of meromorphic functions on $\overline{B}$. We denote by ${\rm Mer}(\overline{B})|_B$ the restriction of ${\rm Mer}(\overline{B})$ onto $B$. Then
$\sigma _B^{*}\left( {\rm Mer}(\overline{B})|_B\right)$ is a {\rm W}-type subfield (\cite{ref4}). Similarly we define ${\rm Mer}(\overline{X})$ and ${\rm Mer}(\overline{X})|_X$. We have a representation of functions in ${\rm Mer}(\overline{B})$ (see Section 6 in \cite{ref1} or Section 6.7.3 in \cite{ref3}).
Any function in ${\rm Mer}(\overline{X})$ has a similar representation. The field $\sigma _X^{*}\left({\rm Mer}(\overline{X})|_X\right)$ is finitely generated over ${\mathbb C}$, non-degenerate and of transcendence degree $n$. Furthermore, it admits an algebraic addition theorem. Therefore it is isomorphic to a {\rm W}-type subfield (Theorem 2.5 in \cite{ref5}). Since the Remmert-Morimoto decomposition is unique, we obtain
\[
{\rm Mer}(\overline{X})|_X \cong {\rm Mer}(\overline{B})|_B.
\]
\par
This remark shows that instead of the canonical form, it is sufficient to consider ${\rm Mer}(\overline{X})|_X$ for $X = {\mathbb C}^n/\Gamma $ with period matrix (2.1) in the proof of Proposition 7.1 in \cite{ref2}.

\section{Singular curves of genus $\pi = 2$}
Let $X = {\mathbb C}/\Lambda $ be an elliptic curve with period matrix
$(1\quad \tau )$, where ${\mathfrak Im}(\tau ) > 0$. Take a subset
$S = \{ P_1, P_2\}$ of $X$ with $P_1 \not= P_2$. Let ${\mathfrak m}$ be a modulus with
support $S$ defined by ${\mathfrak m}(P_1) = {\mathfrak m}(P_2) = 1$.
The relation $R$ on $S$ is the identification of $P_1$ and $P_2$.
Then $\overline{S} := S/R = \{ Q\}$. We construct a singular curve
$X_{{\mathfrak m}}$ with an ordinal double point
$Q$ from ${\mathfrak m}$.
Let $\{ \alpha , \beta \}$ be a homology basis of $X$. We may assume that
closed curves $\alpha $ and $\beta $ do not pass through $P_1$ and $P_2$.
We take small circles $\gamma _1$ and $\gamma _2$ centered at $P_1$ and
$P_2$ with radii $\delta > 0$ and $\varepsilon > 0$ respectively.
Here these circles have anticlockwise direction. We consider a simply
connected domain $X^{\circ}$ obtained by cutting $X$ open along $\alpha $
and $\beta $. We set $X_S^{\circ} := X^{\circ} \setminus \{ P_1, P_2\}$.
Let $\rho : X \longrightarrow X_{{\mathfrak m}}$ be the projection.
Take a basis $\{ \omega , \eta \}$ of $H^0(X_{{\mathfrak m}}, \Omega _{{\mathfrak m}})$ such that
$\rho ^{*}\omega $ is a holomorphic 1-form on $X$. We may assume that
$\omega $ and $\eta $ are normalized as
\begin{equation}
\begin{pmatrix}
\int _{\gamma _1}\rho ^{*}\omega & \int _{\alpha }\rho ^{*}\omega & \int _{\beta }\rho ^{*}\omega \\
\int _{\gamma _1}\rho ^{*}\eta & \int _{\alpha }\rho ^{*}\eta & \int _{\beta }\rho ^{*}\eta 
\end{pmatrix}
=
\begin{pmatrix}
0 & 1 & \tau \\
1 & r_1 & r_2 \\
\end{pmatrix},
\end{equation}
where $r_1, r_2 \in {\mathbb R}$. Let $\Gamma $ be a discrete subgroup of ${\mathbb C}^2$
generated by the above three column vectors. The quotient group ${\mathbb C}^2/\Gamma $
is a toroidal group if and only if $(r_1\quad r_2)$ satisfies the irrationality condition.
When the irrationality condition is not satisfied, we have
${\mathbb C}^2/\Gamma \cong {\mathbb C}^{*} \times {\mathbb T}$, where ${\mathbb T}$ is a
1-dimensional complex torus. In both cases we may consider ${\mathbb C}^2/\Gamma $ as
an analytic Albanese variety of $X_{{\mathfrak m}}$ as shown in the previous section.
Then we set $A = {\rm Alb}^{an}(X_{{\mathfrak m}}) = {\mathbb C}^2/\Gamma $.
The matrix $(1\quad \tau )$ in (3.1) gives a complex torus ${\mathbb T}_{\tau }$. 
We note ${\mathbb T}_{\tau } = X$. The analytic Albanese variety $A$ has the structure of
principal ${\mathbb C}^{*}$-bundle $\sigma : A \longrightarrow {\mathbb T}_{\tau }$
over ${\mathbb T}_{\tau }$. Let $\overline{\sigma } : \overline{A} \longrightarrow {\mathbb T}_{\tau }$
be the associated ${\mathbb P}^1$-bundle. If $A$ is toroidal, then $\overline{A}$ is the
standard compactification of a quasi-abelian variety $A$ of kind 0. If $A$ is not
toroidal, then we have $A \cong {\mathbb C}^{*} \times {\mathbb T} =: B$.
The standard compactification of $B$ is $\overline{B} = {\mathbb P}^1 \times {\mathbb T}$.
As shown in the previous section, we have
${\rm Mer}(\overline{A})|_A \cong {\rm Mer}(\overline{B})|_B$.

\section{Theta functions}
We write ${\bf e}(x) := \exp (2\pi \sqrt{-1}x)$. We consider theta functions for ${\mathbb T}_{\tau }$.
We first define a theta function $\theta (z, \tau )$ by
\begin{equation}
\theta (z, \tau ) := \sum _{n \in {\mathbb Z}} {\bf e}\left(
\frac{1}{2} n^2 \tau + n z\right),\quad z \in {\mathbb C}.
\end{equation}
It satisfies the following equations
\begin{equation}
\left\{ \begin{aligned}
\theta (z + 1, \tau ) &= \theta (z, \tau ),\\
\theta (z + \tau , \tau ) &= {\bf e}\left( - \frac{1}{2}\tau - z\right) \theta (z, \tau ).
\end{aligned}\right.
\end{equation}
Next we define theta functions with characteristics. For any $a, b \in {\mathbb R}$ we define
\begin{equation}
\theta {a \brack b}(z, \tau ) := \sum _{n \in {\mathbb Z}} {\bf e}\left[
\frac{1}{2}(n + a)^2 \tau + (n + a)(z + b)\right].
\end{equation}
Then the function $\theta (z, \tau )$ in (4.1) is written as
$\theta (z, \tau ) = \theta {0 \brack 0}(z, \tau )$.
We have the translation relation
\begin{equation}
\theta {a \brack b}(z + p + q\tau , \tau ) = {\bf e}\left[ - \frac{1}{2}q^2 \tau - q(z + b) + ap\right]
\theta {a \brack b}(z, \tau )
\end{equation}
for any $p, q \in {\mathbb Z}$.\par
We define a theta factor $\rho _0 : \Lambda \times {\mathbb C} \longrightarrow {\mathbb C}^{*}$
by $\rho _0(1, z) = 1$ and $ \rho _0(\tau , z) = {\bf e}\left( - \frac{1}{2}\tau - z\right)$.
We denote by $AF(\rho _0)$ the set of all automorphic forms for $\rho _0$. Then we have
$\theta {0 \brack 0}(z, \tau ) \in AF(\rho _0)$. Since $\dim AF(\rho _0) = 1$, 
$AF(\rho _0)$ is generated by $\theta {0 \brack 0}(z, \tau )$.
We define an into isomorphism $\iota : \Lambda \longrightarrow \Gamma $ by
\begin{equation*}
\iota (\lambda ) := p 
\begin{pmatrix}
1\\
r_1
\end{pmatrix}
+ q \begin{pmatrix}
\tau \\
r_2
\end{pmatrix}
\end{equation*}
for any $\lambda = p + q \tau \in \Lambda $.  We take coordinates
$(z,w)$ of ${\mathbb C}^2$ which represent a period matrix of $\Gamma $ as (3.1).
For any $\gamma \in \Gamma $, $(\gamma _z, \gamma _w)$ is the representation of
$\gamma $ in these coordinates $(z,w)$. We have $\iota (\lambda )_w = pr_1 + qr_2$
for $\lambda = p + q\tau $. Therefore, we can define a homomorphism
$\psi : \Lambda \longrightarrow {\mathbb C}_1^{\times } = \{ \zeta \in {\mathbb C} ;
|\zeta | = 1\}$ by $\psi (\lambda ) := {\bf e}\left( \iota (\lambda )_w\right)$ for
$\lambda \in \Lambda $. For any $\alpha \in {\mathbb Z}$,
$\psi ^{-\alpha }\rho _0$ is also a theta factor. The space $AF(\psi ^{-\alpha }\rho _0)$
of all automorphic forms for $\psi ^{-\alpha }\rho _0$ has the same dimension
as $AF(\rho _0)$. If we set
\begin{equation}
f(z,w) = \frac{\sum _{\text{finite}} D_{\beta }(z) {\bf e}(w)^{\beta }}
{\sum _{\text{finite}} C_{\alpha }(z) {\bf e}(w)^{\alpha }}
\end{equation}
with $C_{\alpha }(z) \in AF(\psi ^{-\alpha }\rho _0)$ and $D_{\beta }(z) \in
AF(\psi ^{-\beta } \rho _0)$, then $f \in \sigma _A^{*}\left( {\rm Mer}(\overline{A})|_A\right)$
(Section 5 in \cite{ref1}, see also Section 6.7 in \cite{ref3}).
Using $\psi ^{-\alpha }(1) = {\bf e}(-\alpha r_1)$, $\psi ^{-\alpha }(\tau ) = {\bf e}(-\alpha r_2)$
and the translation relation of theta functions, we obtain
$\theta {-\alpha r_1 \brack \alpha r_2}(z, \tau ) \in AF(\psi ^{-\alpha }\rho _0)$ for any
$\alpha \in {\mathbb Z}$. Since $\dim AF(\psi ^{-\alpha }\rho _0) = 1$, the space
$AF(\psi ^{-\alpha }\rho _0)$ is generated by $\theta {-\alpha r_1 \brack \alpha r_2}(z, \tau )$.
\par
The maximal complex linear subspace ${\mathbb C}_{\Gamma }^1$ contained in
${\mathbb R}_{\Gamma }^3$ is the $z$-plane. Let $\sigma _{\Gamma } : {\mathbb C}^2
\longrightarrow {\mathbb C}_{\Gamma }^1$ be the projection. We define a factor of
automorphy $\tilde{\rho _0} : \Gamma \times {\mathbb C}^2 \longrightarrow {\mathbb C}^{*}$ by
$\tilde{\rho _0} := \rho _0 \circ (\sigma _{\Gamma } \times \sigma _{\Gamma })$.
Let $AF(\tilde{\rho _0})$ be the set of all automorphic forms for $\tilde{\rho _0}$.
Any $\varphi \in AF(\tilde{\rho _0})$ is represented as
\begin{equation}
\varphi (z, w) = \sum _{\alpha \in {\mathbb Z}} C_{\alpha }(z){\bf e}(w)^{\alpha },\quad
C_{\alpha } \in AF(\psi ^{-\alpha }\rho _0)
\end{equation}
with $\lim _{|\alpha | \to \infty } \root{|\alpha |}\of{|C_{\alpha }(z)|} = 0$ uniformly on
every compact subsets of ${\mathbb C}_{\Gamma }^1$.
We set 
\begin{equation*}
AF_0(\tilde{\rho _0}) := \{ \varphi \in AF(\tilde{\rho _0})\ \text{with finite terms in (4.6)} \}.
\end{equation*}
The space $AF_0(\tilde{\rho _0})$ is generated by
\begin{equation*}
\left\{ \theta {-\alpha r_1 \brack \alpha r_2}(z, \tau ){\bf e}(w)^{\alpha } ;
\alpha \in {\mathbb Z} \right\}.
\end{equation*}
Then the following function is fundamental:
\begin{equation}
\Theta (z,w) := \theta {0 \brack 0}(z, \tau) + \theta {-r_1 \brack r_2}(z, \tau )
{\bf e}(w).
\end{equation}
\par
Fixing $P_0 \in X_S^{\circ }$, we define a period map
$\varphi = (\varphi _1, \varphi _2)$ on $X\setminus S$ by
\begin{equation*}
\varphi _1(P) := \int _{P_0}^{P} \rho ^{*}\omega \quad \text{and}\quad
\varphi _2(P) := \int _{P_0}^{P} \rho ^{*}\eta .
\end{equation*}
Omitting $\tau $, we write $\theta {0 \brack 0}(z) = \theta {0 \brack 0}(z, \tau )$
without confusion. For any $c = (c_1, c_2) \in {\mathbb C}^2$ we define
\begin{equation}
{\mathfrak T}_c(P) := \Theta (\varphi _1(P) - c_1, \varphi _2(P) - c_2).
\end{equation}
Then ${\mathfrak T}_c(P)$ is a multi valued holomorphic function on 
$X \setminus S$. By the proof of Proposition 7.1 in \cite{ref2},
we see that ${\mathfrak T}_c(P)$ is meromorphic on $X$.
It has a simple pole at $P_2$ and is holomorphic on 
$X \setminus \{ P_2\}$. We note
$\lim _{P \to P_1} {\bf e}(\varphi _2(P) - c_2) = 0$.

\section{Generalization of Riemann's constants}
Take $c = (c_1, c_2) \in {\mathbb C}^2$ such that
$\theta {0 \brack 0}(\varphi _1(P_1) - c_1) \not= 0$.
Since $\theta { -r_1 \brack r_2}(\varphi _1(P) - c_1)$ is not identically zero
and $P_2$ is the pole of ${\bf e}(\varphi _2(P) - c_2)$, 
${\mathfrak T}_c(P)$  is not identically zero.
Let $n({\mathfrak T}_c)$ be the number of zeros of ${\mathfrak T}_c$.
The point $P_2$ is the only pole of ${\mathfrak T}_c$. Then we obtain
\begin{equation}
 \begin{aligned}
n({\mathfrak T}_c) - 1 & = \frac{1}{2\pi \sqrt{-1}}\int _{\partial X^{\circ }}
d\log {\mathfrak T}_c\\
 & = \frac{1}{2\pi \sqrt{-1}} \left( \int _{\alpha } + \int _{\beta } +
\int _{\alpha ^{-1}} + \int _{\beta ^{-1}}\right)
d \log {\mathfrak T}_c
 \end{aligned}
\end{equation}
by the argument principle. Let $\alpha ^{-}$ and $\beta ^{-}$ be the
corresponding edges of $X^{\circ }$ to $\alpha $ and $\beta $ respectively.
For $P$ on $\alpha $, we denote by $P^{-}$ the corresponding point to
$P$ on $\alpha ^{-}$. In this case we have
\begin{equation}
\varphi (P^{-}) = \varphi (P) + (\tau , r_2).
\end{equation}
Similarly, we have
\begin{equation}
\varphi (P^{-}) = \varphi (P) - (1, r_1)
\end{equation}
for $P$ on $\beta $. By (5.2) and (5.3) we obtain
\begin{equation}
{\mathfrak T}_c(P^{-}) = {\bf e}\left[ - \frac{1}{2}\tau - (\varphi _1(P) - c_1)\right]
{\mathfrak T}_c(P)
\end{equation}
if $P$ is on $\alpha $, and
\begin{equation}
{\mathfrak T}_c(P^{-}) = {\mathfrak T}_c(P)
\end{equation}
if $P$ is on $\beta $. Therefore, the right hand side of (5.1) is
\begin{equation*}
\frac{1}{2\pi \sqrt{-1}}\left[ \int _{\alpha }d \log {\mathfrak T}_c -
\int _{\alpha }( - 2\pi \sqrt{-1} d \varphi _1 + d \log {\mathfrak T}_c)\right]
 = \int _{\alpha } \rho ^{*}\omega = 1.
\end{equation*}
Hence we have $n({\mathfrak T}_c) = 2$.\par
Let $U_1$ and $U_2$ be open disks enclosed by $\gamma _1$ and $\gamma _2$
respectively. Since $P_1$ and $P_2$ are not zeros of ${\mathfrak T}_c$, the zeros
of ${\mathfrak T}_c$ are in $X^{\circ } \setminus (\overline{U_1} \cup \overline{U_2})$
if $\varepsilon > 0$ and $\delta > 0$ are sufficiently small. Let
$Q_1, Q_2 \in X^{\circ } \setminus (\overline{U_1} \cup \overline{U_2})$ be the
zeors of ${\mathfrak T}_c$. We define a divisor
$D := Q_1 + Q_2 \in {\rm Div}(X_{{\mathfrak m}})$, where ${\rm Div}(X_{{\mathfrak m}})$ is
the group of divisors prime to $S$ (see \cite{ref2}).\par
We determine the image $W := \varphi (D)$ of $D$. Let $t$ be a coordinate
on $\overline{U_2}$ with $t(P_2) = 0$. Then ${\mathfrak T}_c(P)$ is written as a function
of $t$ as follows:
\begin{equation}
{\mathfrak T}_c(t) = \frac{c_{-1}}{t} + h_2(t; c),
\end{equation}
where $c_{-1} \not= 0$ and $h_2(t;c)$ is a holomorphic function of $t$ depending
on $c$. Therefore we have
\begin{equation}
\frac{{\mathfrak T}_c'(t)}{{\mathfrak T}_c(t)} = - \frac{1}{t} + h_3(t; c),
\end{equation}
where $h_3(t; c)$ is a holomorphic function of $t$ depending on $c$. We set
$H_3(t; c) := \int _{0}^{t} h_3(t; c)dt$. We define a map
$d(t) : {\mathbb C}^2 \longrightarrow {\mathbb C}^2$ by
\begin{equation}
d(t)(c) = (d_1(c), d_2(t)(c)) := \left( c_1, c_1r_1 + \frac{1}{2\pi \sqrt{-1}} H_3(t; c)\right)
\end{equation}
for $c = (c_1, c_2) \in {\mathbb C}^2$.\par
The circle $\gamma _2$ is written as $t = \varepsilon {\bf e}(u)$,
$0\leq u \leq 1$. We set $a(\varepsilon ) := \int _{0}^{1}\varphi _2(\varepsilon {\bf e}(u))du$.
We note that $a(\varepsilon )$ does not depend on $c = (c_1, c_2)$.
We may assume that $\alpha $ and $\beta $ have the common initial
point $Q_0$. We define a generalized Riemann's constant $\kappa (\varepsilon )=
(\kappa _1, \kappa _2(\varepsilon ))$ by
\begin{align*}
\kappa _1 &:= - \frac{1}{2}\tau - \varphi _1(Q_0) + \varphi _1(P_2) +
\int _{\alpha }\varphi _1 \rho ^{*}\omega ,\\
\kappa _2(\varepsilon ) & := \left( - \frac{1}{2}\tau - \varphi _1(Q_0)\right) r_1 +
a(\varepsilon ) + \int _{\alpha }\varphi _2 \rho ^{*}\omega .
\end{align*}

\begin{theorem}
Using the above notations, we obtain
\begin{equation}
W = \varphi (D) \equiv d(\varepsilon )(c) + \kappa (\varepsilon )\  \mod \Gamma .
\end{equation}
\end{theorem}

\begin{proof}
By the argument principle we obtain
\begin{equation}
\begin{aligned}
W &= \frac{1}{2\pi \sqrt{-1}}\int _{\partial (X^{\circ } \setminus (\overline{U_1} \cup
\overline{U_2}))}\varphi d \log {\mathfrak T}_c\\
& = \frac{1}{2\pi \sqrt{-1}}\left( \int _{\alpha } + \int _{\beta } + \int _{\alpha ^{-1}}
+ \int _{\beta ^{-1}} - \int _{\gamma _1} - \int _{\gamma _2}\right) \varphi d \log
{\mathfrak T}_c.
\end{aligned}
\end{equation}
\par
If $P$ is on $\alpha $, then we define $\varphi ^{-}(P) := \varphi (P^{-})$ and
${\mathfrak T}_c^{-}(P) := {\mathfrak T}_c(P^{-})$, where $P^{-}$ is the corresponding
point to $P$ on $\alpha ^{-}$. We consider
\begin{multline}
\frac{1}{2\pi \sqrt{-1}}\left( \int _{\alpha } + \int _{\alpha ^{-1}}\right) \varphi d \log
{\mathfrak T}_c\\
= \frac{1}{2\pi \sqrt{-1}}\int _{\alpha }(\varphi - \varphi ^{-})d \log {\mathfrak T}_c +
\frac{1}{2\pi \sqrt{-1}}\int _{\alpha } \varphi ^{-} ( d \log {\mathfrak T}_c -
d \log {\mathfrak T}_c^{-}).
\end{multline}
By (5.2) and (5.4) we have
\begin{gather*}
\varphi (P) - \varphi ^{-}(P) = - (\tau , r_2),\\
\log {\mathfrak T}_c^{-}(P) = 2\pi \sqrt{-1}\left[ - \frac{1}{2} \tau
- (\varphi _1(P) - c_1)\right] + \log {\mathfrak T}_c(P)
\end{gather*}
for $P$ on $\alpha $. Therefore we obtain
\begin{equation}
\begin{split}
\frac{1}{2\pi \sqrt{-1}}\left( \int _{\alpha } \right. & \left. + \int _{\alpha ^{-1}}\right) \varphi
d \log {\mathfrak T}_c\\
& 
\equiv \int _{\alpha }\varphi d\varphi _1 - \left( \frac{1}{2\pi \sqrt{-1}}\int_{\alpha }
d \log {\mathfrak T}_c\right) (\tau , r_2)\ \mod \Gamma
\end{split}
\end{equation}
by (5.11). Let $Q'$ be the terminal point of $\alpha $. Then we have
$\varphi (Q') = \varphi (Q_0) + (1, r_1)$ and
\begin{equation*}
\begin{aligned}
{\mathfrak T}_c(Q') & = \theta {0 \brack 0}(\varphi _1(Q_0) + 1 - c_1) +
\theta {-r_1 \brack r_2}(\varphi _1(Q_0) + 1 - c_1) {\bf e}(\varphi _2(Q_0) +
r_1 - c_2)\\
 & = {\mathfrak T}_c(Q_0).
\end{aligned}
\end{equation*}
Therefore we obtain
\begin{equation}
\int _{\alpha }d \log {\mathfrak T}_c = 0.
\end{equation}
From (5.12) and (5.13) it follows that
\begin{equation}
\frac{1}{2\pi \sqrt{-1}}\left( \int _{\alpha } + \int _{\alpha ^{-1}}\right) \varphi d \log
{\mathfrak T}c \equiv \int _{\alpha }\varphi d \varphi _1 \ \mod \Gamma .
\end{equation}
\par
We carry out a similar calculation on $\beta $. We have
\begin{equation}
\begin{split}
\frac{1}{2\pi \sqrt{-1}}\left( \int _{\beta } \right. & \left. + \int _{\beta ^{-1}}\right) \varphi
d \log {\mathfrak T}_c\\
& = \int _{\beta }(\varphi - \varphi ^{-})d \log {\mathfrak T}_c + \frac{1}{2\pi \sqrt{-1}}
\int _{\beta } \varphi ^{-}(d \log {\mathfrak T}_c - d \log {\mathfrak T}_c^{-}),
\end{split}
\end{equation}
where $\varphi ^{-}$ and ${\mathfrak T}_c^{-}$ are similarly defined on $\beta $.
By (5.3) and (5.5) we obtain $\varphi - \varphi ^{-} = (1, r_1)$ and
${\mathfrak T}_c = {\mathfrak T}_c^{-}$ on $\beta $. Hence the right hand side of (5.15) is
\begin{equation*}
\left(\frac{1}{2\pi \sqrt{-1}}\int _{\beta } d \log {\mathfrak T}_c\right) (1, r_1).
\end{equation*}
Let $Q''$ be the terminal point of $\beta $. Since
$\varphi (Q'') = \varphi (Q_0) + (\tau , r_2)$, we obtain
\begin{equation*}
\begin{aligned}
{\mathfrak T}_c(Q'') & = \theta {0 \brack 0}(\varphi _1(Q'') - c_1) +
\theta {-r_1 \brack r_2}(\varphi _1(Q'') - c_1) {\bf e}(\varphi _2(Q'') 
 - c_2)\\
 & = {\bf e}\left[ - \frac{1}{2}\tau - (\varphi _1(Q_0) - c_1)\right] {\mathfrak T}_c(Q_0),
\end{aligned}
\end{equation*}
hence
\begin{equation*}
\begin{aligned}
\frac{1}{2\pi \sqrt{-1}}\int _{\beta }d \log {\mathfrak T}_c & =
\frac{1}{2\pi \sqrt{-1}}\left( \log {\mathfrak T}_c(Q'') - \log {\mathfrak T}_c(Q_0)\right)\\
& = - \frac{1}{2}\tau - (\varphi _1(Q_0) - c_1).
\end{aligned}
\end{equation*}
Thus we obtain
\begin{equation}
\begin{split}
\frac{1}{2\pi \sqrt{-1}}\left( \int _{\beta }  \right. & + \left. \int _{\beta ^{-1}}\right)
\varphi d \log {\mathfrak T}_c \\
& = \left[ - \frac{1}{2}\tau - ( \varphi _1(Q_0) - c_1)\right] (1, r_1).
\end{split}
\end{equation}
\par
Next we consider $\int _{\gamma _j} \varphi d \log {\mathfrak T}_C$ for $j = 1, 2$.
Since $\varphi _1(P)$ is holomorphic at $P_1$ and ${\mathfrak T}_c(P_1) \not= 0$,
$\varphi _1 d \log {\mathfrak T}_c$ is holomorphic at $P_1$. Then we have
\begin{equation}
\int _{\gamma _1} \varphi _1 d \log {\mathfrak T}_c = 0.
\end{equation}
Furthermore, we obtain
\begin{equation}
\frac{1}{2\pi \sqrt{-1}} \int _{\gamma _2} \varphi _1 d \log {\mathfrak T}_c =
- \varphi _1(P_2),
\end{equation}
for ${\mathfrak T}_c$ has the pole of order 1 at $P_2$. We note
\begin{equation}
\int _{\gamma _1} \varphi _2 d \log {\mathfrak T}_c = \int _{\gamma _1}
d(\varphi _2 \log {\mathfrak T}_c) - \int _{\gamma _1}(\log {\mathfrak T}_c)d 
\varphi _2 .
\end{equation}
Let $t$ be a local coordinate on $\overline{U_1}$ with $t(P_1) = 0$.
Since $d \varphi _2 = \rho ^{*}\eta $ has a pole of order 1 at $P_1$, we have the
following representation
\begin{equation*}
d \varphi _2 = \left( \frac{1}{2\pi \sqrt{-1}} \frac{1}{t} + h(t)\right) dt,
\end{equation*}
where $h(t)$ is a holomorphic function. There exists a single-valued branch of
$\log {\mathfrak T}_c$ on $\overline{U_1}$. Then we obtain
\begin{equation}
\begin{split}
\int _{\gamma _1}(\log {\mathfrak T}_c)d \varphi _2 & = \frac{1}{2\pi \sqrt{-1}}
\int _{\gamma _1} (\log {\mathfrak T}_c) \frac{1}{t}dt\\
& = \log {\mathfrak T}_c(P_1).
\end{split}
\end{equation}
The circle $\gamma _1$ is parametrized as $t = \delta {\bf e}(u)$, $0 \leq u \leq 1$.
Then we have
\begin{equation*}
\begin{split}
\int _{\gamma _1} d( \varphi _2 \log {\mathfrak T}_c) & = \varphi _2(\delta {\bf e}(1))
\log {\mathfrak T}_c(\delta {\bf e}(1)) - \varphi _2(\delta {\bf e}(0))
\log {\mathfrak T}_c(\delta {\bf e}(0)) \\
& = \left( \int _{\gamma _1}\rho ^{*}\eta \right) \log {\mathfrak T}_c(\delta)\\
& =  \log {\mathfrak T}_c(\delta).
\end{split}
\end{equation*}
Therefore, we obtain
\begin{equation}
\int _{\gamma _1} d(\varphi _2 \log {\mathfrak T}_c) = \log {\mathfrak T}_c(\delta )
\end{equation}
by (5.19). It follows from (5.19), (5.20) and (5.21) that
\begin{equation}
\lim _{\delta \to 0} \int _{\gamma _1} \varphi _2 d \log {\mathfrak T}_c = 0.
\end{equation}
\par
The functions ${\mathfrak T}_c(t)$ and ${\mathfrak T}_c'(t)/{\mathfrak T}_c(t)$ are represented as in
(5.6) and (5.7) respectively. Then we obtain
\begin{equation}
\int _{\gamma _2} \varphi _2 d \log {\mathfrak T}_c = - \int _{\gamma _2} \varphi _2(t)
\frac{1}{t} dt + \int _{\gamma _2} \varphi _2(t) h_3(t; c)dt.
\end{equation}
If we parametrize $\gamma _2$ as $t = \varepsilon {\bf e}(u)$, $0 \leq u \leq 1$, then we have
\begin{equation}
\int _{\gamma _2} \varphi _2(t)\frac{1}{t}dt = 2\pi \sqrt{-1}a(\varepsilon ).
\end{equation}
Since $d(\varphi _2H_3) = H_3d \varphi _2 + \varphi _2 d H_3$, we have
\begin{equation}
\int _{\gamma _2} \varphi _2(t) h_3(t;c)dt = \int _{\gamma _2}d(\varphi _2 H_3) -
\int _{\gamma _2}H_3 \rho ^{*}\eta .
\end{equation}
Furthermore we have
\begin{equation*}
\begin{split}
\int _{\gamma _2}d(\varphi _2 H_3) & = \varphi _2(\varepsilon {\bf e}(1))
H_3(\varepsilon {\bf e}(1);c) - \varphi _2(\varepsilon {\bf e}(0))
H_3(\varepsilon {\bf e}(0);c)\\
& = \left( \int _{\gamma _2}\rho ^{*}\eta \right) H_3(\varepsilon ; c)\\
& = - H_3(\varepsilon ; c).
\end{split}
\end{equation*}
We have the representation of $\rho ^{*}\eta $ on $\overline{U_2}$ as
\begin{equation*}
\rho ^{*}\eta = \left( - \frac{1}{2\pi \sqrt{-1}} \frac{1}{t} + h_1(t)\right)dt,
\end{equation*}
where $h_1(t)$ is a holomorphic function. Then we obtain
\begin{equation*}
\begin{split}
\int _{\gamma _2}H_3 \rho ^{*}\eta & = - \frac{1}{2\pi \sqrt{-1}}\int _{\gamma _2}
H_3(t; c)\frac{1}{t}dt\\
& = - H_3(0; c) = 0.
\end{split}
\end{equation*}
Therefore we have
\begin{equation*}
\int _{\gamma _2}\varphi _2(t) h_3(t; c) dt = - H_3(\varepsilon : c)
\end{equation*}
by (5.25). Hence we obtain
\begin{equation}
\int _{\gamma _2} \varphi _2 d \log {\mathfrak T}_c = - 2\pi \sqrt{-1} a(\varepsilon )
- H_3(\varepsilon ; c).
\end{equation}
Letting $\delta \to 0$, we finally obtain
\begin{equation*}
\begin{split}
W & \equiv \int _{\alpha }\varphi \rho ^{*}\omega + \left[ - \frac{2}{2}\tau - 
(\varphi _1(Q_0) - c_1)\right] (1, r_1)\\
& \quad + \left( \varphi _1(P_2), a(\varepsilon ) + \frac{1}{2\pi \sqrt{-1}}
H_3(\varepsilon ; c)\right)\\
& = d(\varepsilon )(c) + \kappa (\varepsilon ) \ \mod \Gamma
\end{split}
\end{equation*}
by (5.10), (5.14), (5.16), (5.17), (5.18), (5.22) and (5.26).
\end{proof}

\section{The image of $X \setminus S$} 
Let $t$ be the local coordinate on $\overline{U_2}$ in the previous section.
We consider ${\mathfrak T}_c(P)$ as a function of $t$ on $\overline{U_2}$
as follows:
\begin{equation*}
{\mathfrak T}_c(t) = \theta {0 \brack 0}(\varphi _1(t) - c_1) +
\theta {-r_1 \brack r_2} (\varphi _1(t) - c_1){\bf e}(\varphi _2(t) - c_2).
\end{equation*}
By a straight calculation we obtain
\begin{equation}
h_3(t; c) = \frac{h_2(t; c) + h_2'(t; c)t}{c_{-1} + h_2(t; c)t}
\end{equation}
by (5.6) and (5.7), where $h_2'(t; c)$ is the derivative of $h_2(t; c)$ by $t$.
Especially we have
\begin{equation}
h_3(0; c) = \frac{1}{c_{-1}} h_2(0; c).
\end{equation}
\par
We determine $c_{-1}$ and $h_2(t; c)$. Fix a point $\tilde{P_2} \in U_2 \setminus
\{ P_2\}$ and set $t_0 := t(\tilde{P_2})$. Then we have
\begin{equation*}
{\bf e}(\varphi _2(t)) = {\bf e}\left( \varphi _2(t_0) + \int _{t_0}^{t}\rho ^{*}\eta \right).
\end{equation*}
Using the representation of $\rho ^{*}\eta $ on $\overline{U_2}$ in the previous section, 
we obtain
\begin{equation*}
{\bf e}(\varphi _2(t)) = {\bf e}\left( \varphi _2(t_0) + \log t_0 + \int _{t_0}^{t}
h_1(t) dt \right) \frac{1}{t}.
\end{equation*}
If we set
\begin{equation*}
g(t) = {\bf e}\left( \varphi _2(t_0) + \log t_0 + \int _{t_0}^{t} h_1(t)dt\right),
\end{equation*}
then we have
\begin{equation*}
{\mathfrak T}_c(t) = \theta {0 \brack 0}(\varphi _1(t) - c_1) + \theta {-r_1 \brack r_2}
(\varphi _1(t) - c_1) g(t) {\bf e}(- c_2) \frac{1}{t}.
\end{equation*}
Therefore we obtain
\begin{equation}
c_{-1} = \theta {-r_1 \brack r_2}(\varphi _1(P_2) - c_1) g(0) {\bf e}(- c_2)
\end{equation}
and
\begin{equation}
\begin{split}
h_2(t; c) &= \theta {0 \brack 0}(\varphi _1(t) - c_1)\\
& \quad + \left\{ \theta {-r_1 \brack r_2}(\varphi _1(t) - c_1) g(t) -
\theta {-r_1 \brack r_2}(\varphi _1(P_2) - c_1) g(0) \right\} {\bf e}( - c_2).
\end{split}
\end{equation}
Hence we obtain
\begin{equation}
h_3(0; c) = \frac{\theta {0 \brack 0}(\varphi _1(P_2) - c_1){\bf e}(c_2)}
{\theta {-r_1 \brack r_2}(\varphi _1(P_2) - c_1) g(0)}
\end{equation}
by (6.2), (6.3) and (6.4).\par
The Jacobian matrix $J_{d(t)}$ of the map $d(t)$ is
\begin{equation}
\begin{split}
J_{d(t)}(c) & =
\begin{pmatrix}
\frac{\partial d_1}{\partial c_1}(c) & \frac{\partial d_1}{\partial c_2}(c)\\
\frac{\partial d_2(t)}{\partial c_1}(c) & \frac{\partial d_2(t)}{\partial c_2}(c)
\end{pmatrix}\\
& = 
\begin{pmatrix}
1 & 0\\
\frac{\partial d_2(t)}{\partial c_1}(c) & \frac{1}{2\pi \sqrt{-1}}
\frac{\partial H_3(t; c)}{\partial c_2}
\end{pmatrix}.
\end{split}
\end{equation}
Then, if we set $N_t := \{ c \in {\mathbb C}^2 ; \partial H_3(t; c)/\partial c_2 = 0 \}$,
then $d(t)$ is a locally biholomorphic map on ${\mathbb C}^2 \setminus N_t$.
We define a map $d'(t) : {\mathbb C}^2 \longrightarrow {\mathbb C}^2$ by
\begin{equation*}
d'(t)(c) = (d'_1(c), d'_2(t)(c)) := \left( c_1, \frac{1}{2\pi \sqrt{-1}}h_3(t; c)\right)
\end{equation*}
for any $c = (c_1, c_2) \in {\mathbb C}^2$. For any $t \in \overline{U_2}$
we define the period group $\widetilde{\Gamma }_t$ of $d(t)$ by
\begin{equation*}
\widetilde{\Gamma }_t := \{ \tilde{c} \in {\mathbb C}^2 ; d(t)(c + \tilde{c}) = d(t)(c)\ 
\text{for all}\ c \in {\mathbb C}^2\}.
\end{equation*}
Similarly,
\begin{equation*}
\widetilde{\Gamma }'_t := \{ \tilde{c} \in {\mathbb C}^2 ; d'(t)(c + \tilde{c}) = d'(t)(c)\ 
\text{for all}\ c \in {\mathbb C}^2\}
\end{equation*}
is the period group of $d'(t)$.

\begin{lemma}
There exists a neighbourhood $U_{2,0}$ of $P_2$ relatively compact in $U_2$
such that $\widetilde{\Gamma }'_t = \{ 0 \} \times {\mathbb Z}$ for all
$t \in U_{2,0}$.
\end{lemma}

\begin{proof}
By (6.3) and (6.4) we can write as
\begin{equation*}
h_2(t; c) = \alpha _1(t; c_1) + \alpha _2(t; c_1) {\bf e}(- c_2)
\end{equation*}
and $c_{-1} = \beta (c_1){\bf e}(- c_2)$, where $\alpha _1(t; c)$ and
$\alpha _2(t; c_1)$ are holomorphic in $(t, c_1)$. Furthermore we have
\begin{equation*}
h_2'(t; c) = \alpha _1'(t; c_1) + \alpha _2'(t; c_1){\bf e}( - c_2),
\end{equation*}
where $\alpha _1'(t; c_1)$ and $\alpha _2'(t; c_1)$ are derivatives of
$\alpha _1(t; c_1)$ and $\alpha _2(t; c_1)$ by $t$ respectively.
It follows from (6.1) and the above representations that
\begin{equation}
h_3(t; c) = \frac{A(t; c_1) + B(t; c_1){\bf e}( - c_2)}
{C(t; c_1) + D(t; c_1){\bf e}( - c_2)},
\end{equation}
where $A(t; c_1), B(t; c_1), C(t; c_1)$ and $D(t; c_1)$ are holomorphic in $(t, c_1)$.
By (6.5) we have
\begin{equation*}
\begin{pmatrix}
A(0; c_1) & B(0; c_1)\\
C(0; c_1) & D(0; c_1)
\end{pmatrix} =
\begin{pmatrix}
\theta {0 \brack 0}(\varphi _1(P_2) - c_1) & 0 \\
0 & \theta {-r_1 \brack r_2}(\varphi _1(P_2) - c_1) g(0)
\end{pmatrix}.
\end{equation*}
Then, there exists a relatively compact open neighbourhood $U_{2,0}$ of $P_2$
in $U_2$ such that
\begin{equation}
\begin{vmatrix}
A(t; c_1) & B(t; c_1)\\
C(t; c_1) & D(t; c_1)
\end{vmatrix}
\not= 0
\end{equation}
for all $t \in U_{2,0}$. If we fix $c_1$, then the period group of $h_3(t; c)$ with
respect to $c_2$ is ${\mathbb Z}$. Therefore, we obtain
$\widetilde{\Gamma }_t' = \{ 0 \} \times {\mathbb Z}$ for all $t \in U_{2,0}$.
\end{proof}

Let $\widetilde{\Gamma } := \bigcap _{t \in U_{2,0}}\widetilde{\Gamma }_t$.
We have $\{ 0 \} \times {\mathbb Z} \subset \widetilde{\Gamma }_t$
for any $t \in U_{2,0}$. Then $\{ 0 \} \times {\mathbb Z} \subset \widetilde{\Gamma }$.
Since $h_3(t; c) = \partial H_3(t; c)/\partial t$, we obtain
$\widetilde{\Gamma } \subset \widetilde{\Gamma }_t'$ for any $t \in U_{2,0}$.
Therefore we obtain $\widetilde{\Gamma } = \{ 0 \} \times {\mathbb Z}$ by Lemma 6.1.
\par
We note that $h_3(t; c)$ is not a bounded function with respect to $c_2$ by (6.7),
hence so is $H_3(t; c)$.

\begin{lemma}
Take any $c \in {\mathbb C}^2$ and any $t \in U_{2,0} \setminus \{ P_2\}$. If
\begin{equation}
d(t)(c + \tilde{c}) = d(t)(c),
\end{equation}
then $\tilde{c} = (0, q)$ for some $q \in {\mathbb Q}$.
\end{lemma}

\begin{proof}
Assume that $\tilde{c}$ satisfies (6.9). Since
$d(t)(c) = (d_1(c), d_2(t)(c))$ and $d_1(c) = c_1$, we have
$\tilde{c} = (0, \tilde{c}_2)$ with $\tilde{c}_2 \in {\mathbb C}$. We note that any
$(0, k)$ with $k \in {\mathbb Z}$ satisfies (6.9). If $1$ and $\tilde{c}_2$
are linearly independent over ${\mathbb R}$, then $H_3(t; c)$ is bounded in $c_2$.
Therefore $1$ and $\tilde{c}_2$ are linearly dependent over ${\mathbb R}$.
If $1$ and $\tilde{c}_2$ are linearly independent over ${\mathbb Z}$,
then $H_3(t; c)$ is constant with respect to $c_2$ by the uniqueness theorem.
Hence $1$ and $\tilde{c}_2$ are linearly dependent over ${\mathbb Z}$.
Thus we have $\tilde{c}_2 \in {\mathbb Q}$.
\end{proof}

We denote $\tilde{c}(q) := (0, q)$ for $q \in {\mathbb Q}$. We set
$T := \{ t \in U_{2,0} ; \widetilde{\Gamma } \not=
 \widetilde{\Gamma }_t \}.$

\begin{proposition}
The set $T$ is a nowhere dense subset of $U_{2,0}$.
\end{proposition}

\begin{proof}
Take any $t \in T$. Then there exists $\tilde{c} \in \widetilde{\Gamma }_t \setminus
\widetilde{\Gamma }$. By Lemma 6.2 $\tilde{c} = \tilde{c}(q)$ for some
$q \in {\mathbb Q} \setminus {\mathbb Z}$. Let
\begin{equation*}
T_q := \{ t' \in U_{2,0} ; \tilde{c}(q)\ \text{is a period of $d(t')$}\}.
\end{equation*}
Then $t \in T_q$. However, there exists $t_0 \in U_{2,0}$ such that
$\tilde{c}(q)$ is not a period of $d(t_0)$. Then there exists $c_0 \in {\mathbb C}^2$
such that $d(t_0)(c_0 + \tilde{c}(q)) \not= d(t_0)(c_0)$.
If we set
\begin{equation*}
A := \{ t' \in U_{2,0} ; d(t')(c_0 + \tilde{c}(q)) = d(t')(c_0)\},
\end{equation*}
then $A$ is a discrete subset of $U_{2,0}$. Since $T_q \subset A$, $T_q$ is
also a discrete subset. Hence $T$ is a nowhere dense subset for
$T \subset \bigcup _{q \in {\mathbb Q}\setminus {\mathbb Z}}T_q$.
\end{proof}

For any $\varepsilon _0 > 0$ we can take $\varepsilon $ with
$0 < \varepsilon < \varepsilon _0$ such that there exists
$t_{\varepsilon } \in U_{2,0}$ with $|t_{\varepsilon }| = \varepsilon $ and
$\widetilde{\Gamma }_{t_{\varepsilon }} = \{ 0 \} \times {\mathbb Z}$,
by the above proposition.

\begin{proposition}
Let $t_{\varepsilon } \in U_{2,0}$ be as above. Then, 
$d(t_{\varepsilon }) : {\mathbb C}^2 \setminus N_{t_{\varepsilon }}
\longrightarrow  \Omega $ is a covering space, where 
$\Omega := d(t_{\varepsilon })({\mathbb C}^2 \setminus
N_{t_{\varepsilon }})$.
\end{proposition}

\begin{proof}
Take any $v_0 \in {\mathbb C}^2 \setminus N_{t_{\varepsilon }}$.
We set $u_0 := d(t_{\varepsilon })(v_0)$. There exist a neighbourhood
$V_0$ of $v_0$ and a neighbourhood $U_0$ of $u_0$ such that
$d(t_{\varepsilon })|_{V_0} : V_0 \longrightarrow U_0$ is biholomorphic.
We set $V_k := V_0 + \tilde{c}(k)$ for $k \in {\mathbb Z}$. Since
$\widetilde{\Gamma }_{t_{\varepsilon }} = \{ 0 \} \times {\mathbb Z}$,
$d(t_{\varepsilon })|_{V_k} : V_k \longrightarrow U_0$ is biholomorphic.
\par
It remains to prove that
$d(t_{\varepsilon })^{-1}(U_0) = \bigsqcup _{k \in {\mathbb Z}}V_k$
(disjoint union). It is sufficient to show that for any
$\tilde{v}_0 \in d(t_{\varepsilon })^{-1}(u_0)$ there exists
$k \in {\mathbb Z}$ with $\tilde{v}_0 \in V_k$. Suppose that there exists
$\tilde{v}_0 \in \left( {\mathbb C}^2 \setminus N_{t_{\varepsilon }}\right)
\setminus \left( \bigsqcup _{k \in {\mathbb Z}} V_k\right)$ such that
$d(t_{\varepsilon })(\tilde{v}_0) = u_0$.
Let $\tilde{c} := \tilde{v}_0 - v_0$. Then we have
$d(t_{\varepsilon })(v_0 + \tilde{c}) = d(t_{\varepsilon })(v_0)$. Therefore,
$\tilde{c} = \tilde{c}(q) = (0, q)$ for some $q \in {\mathbb Q}$ by Lemma 6.2.
However, $q \notin {\mathbb Z}$ by the assumption. For any 
$p \in {\mathbb Q} \setminus {\mathbb Z}$ we define
\begin{equation*}
A_p := \{ v \in V_0 ; d(t_{\varepsilon })(v + \tilde{c}(p)) =
d(t_{\varepsilon })(v) \}.
\end{equation*}
Then $A_p$ is an analytic subset of codimension 1.
There exist a neighbourhood $\widetilde{V}_0$ of $\tilde{v}_0$
and a neighbourhood $\widetilde{U}_0$ of $u_0$ with
$\widetilde{U}_0 \subset U_0$ such that
$d(t_{\varepsilon })|_{\widetilde{V}_0} : \widetilde{V}_0 
\longrightarrow \widetilde{U}_0$ is biholomorphic and
$\widetilde{V}_0 \cap V_k = \emptyset $ for any $k \in {\mathbb Z}$.
Since $\bigcup _{p \in {\mathbb Q} \setminus {\mathbb Z}} A_p$
is a thin set, we can take
$\tilde{v} \in \widetilde{V}_0 \setminus \left( \bigcup _{p \in 
{\mathbb Q} \setminus {\mathbb Z}} A_p\right)$. Then, $\tilde{v}$
has the following property:\\
if $d(t_{\varepsilon })(\tilde{v} + \tilde{\tilde{c}}) = d(t_{\varepsilon })
(\tilde{v})$, then $\tilde{\tilde{c}} \in \{ 0 \} \times {\mathbb Z}$.\par
Let $v := \left( d(t_{\varepsilon })|_{V_0}\right)^{-1}\left(
d(t_{\varepsilon })(\tilde{v})\right) \in V_0$.
Then $d(t_{\varepsilon })(v) = d(t_{\varepsilon })(\tilde{v})$.
We set $\tilde{\tilde{c}} := v - \tilde{v}$.
If we take a sufficiently small $\widetilde{V}_0$, then
$\tilde{\tilde{c}} \notin \{ 0 \} \times {\mathbb Z}$. Obviously, we have
$d(t_{\varepsilon })(\tilde{v} + \tilde{\tilde{c}}) = d(t_{\varepsilon })(\tilde{v})$.
Then $\tilde{\tilde{c}}$ must be in $\{ 0 \} \times {\mathbb Z}$ by
the property of $\tilde{v}$. This is a contradiction. Thus we obtain
$d(t_{\varepsilon })^{-1}(U_0) = \bigsqcup _{k \in {\mathbb Z}}V_k$.
\end{proof}

Let $\arg (t_{\varepsilon }) = \theta $ in Proposition 6.4. If we take
a new local coordinate $te^{-\sqrt{-1}\theta }$, then we have
$t_{\varepsilon } = \varepsilon $. Therefore, we may assume that
$d(\varepsilon ) : {\mathbb C}^2 \setminus N_{\varepsilon } \longrightarrow
\Omega := d(\varepsilon ) ({\mathbb C}^2 \setminus N_{\varepsilon })$ 
is a covering space.\par
Take any $u \in \Omega + \kappa (\varepsilon )$. Then, there exists a
neighbourhood $U$ of $u - \kappa (\varepsilon )$ such that
$d(\varepsilon )^{-1}(U) = \bigsqcup _{k \in {\mathbb Z}}V_k$ and
$d(\varepsilon )|_{V_k} : V_k \longrightarrow U$ is biholomorphic.
We define
\begin{equation}
\beta _k(u) := \left( d(\varepsilon )|_{V_k}\right)^{-1}
(u - \kappa (\varepsilon ))
\end{equation}
for any $u \in \Omega + \kappa (\varepsilon )$ and any $k \in {\mathbb Z}$.

\begin{lemma}
Let $\Theta $ be the fundamental function defined in (4.7).
Take any $u \in \Omega + \kappa (\varepsilon )$. 
Then,
$\Theta (u - \beta _k(u)) = 0$ if and only if
$\Theta (u - \beta _{\ell }(u)) = 0$, for any $k, \ell \in {\mathbb Z}$.
\end{lemma}

\begin{proof}
By the definition we have $\beta _k(u) - \beta _{\ell }(u) \in \{ 0 \} \times {\mathbb Z}$.
Since $\Theta (u)$ is invariant by $\{ 0 \} \times {\mathbb Z}$,
the statement holds.
\end{proof}

From the above lemma it follows that
\begin{equation*}
Z(\varepsilon ) := \{ u \in \Omega + \kappa (\varepsilon ) ;
\Theta (u - \beta _k(u)) = 0\ \text{for some $k$}\}
\end{equation*}
is well-defined as an analytic subset of $\Omega + \kappa (\varepsilon )$.
We define the image of $(X \setminus S)\setminus \overline{U_2}$ by
$W_1(\varepsilon ) := \varphi ((X \setminus S) \setminus \overline{U_2})$.

\begin{theorem}
The image $W_1(\varepsilon )$ is contained in $Z(\varepsilon )$.
\end{theorem}

\begin{proof}
Take any $p_1 \in (X \setminus S) \setminus \overline{U_2}$. 
It suffices to show that $\varphi (p_1) \in Z(\varepsilon )$.
Let $\tilde{p}_1$ and $\tilde{p}_2$ be points arbitrarily close to
$p_1$ and $P_0$ respectively. We define a divisor
$D := \tilde{p}_1 + \tilde{p}_2 \in {\rm Div}(X_{{\mathfrak m}})$.
We set $c := \varphi (D) \in \Omega + \kappa (\varepsilon )$.
We may assume that $c$ is a general point. Let $q_1$ and $q_2$
be the zeros of ${\mathfrak T}_{\beta _k(c)}(P) = \Theta (
\varphi (P) - \beta _k(c))$, where $\beta _k(c)$ is defined in (6.10).
If $E := q_1 + q_2$ is the divisor given by $q_1$ and $q_2$, then
we have
\begin{equation*}
\varphi (E) \equiv d(\varepsilon ) (\beta _k(c)) + \kappa (\varepsilon )\ \mod \Gamma
\end{equation*}
by Theorem 5.1. Since
\begin{equation*}
d(\varepsilon )(\beta _k(c)) = d(\varepsilon )\left( \left(d(\varepsilon )|_
{V_k}\right)^{-1}(c - \kappa (\varepsilon ))\right) = c - \kappa (\varepsilon ),
\end{equation*}
we obtain
\begin{equation*}
\varphi (E) \equiv \varphi (D)\ \mod \Gamma .
\end{equation*}
Then $E = D$, for $c$ is general. Hence we have
\begin{equation}
\Theta \left( \varphi (\tilde{p}_1) - \beta _k(\varphi (\tilde{p}_1 ) +
\varphi (\tilde{p}_2)) \right) = 0.
\end{equation}
We note $\varphi (P_0) = (0, 0)$. Letting $\tilde{p}_1 \to p_1$ and
$\tilde{p}_2 \to P_0$ in (6.11), we obtain
\begin{equation*}
\Theta \left( \varphi (p_1) - \beta _k(\varphi (p_1))\right) = 0.
\end{equation*}
Thus we have $\varphi (p_1) \in Z(\varepsilon )$.
\end{proof}


\begin{thebibliography}{99}


\bibitem{ref1}
\textit{Y.~Abe},
Explicit representation of degenerate abelian functions and related topics, Far East J. Math. Sci. \textbf{70} (2012), 321--336.

\bibitem{ref2}
\textit{Y.~Abe},
Analytic study of singular curves, preprint 2017, arxive:1609.04517.

\bibitem{ref3}
\textit{Y.~Abe}, 
Toroidal Groups, Yokohama Publishers, Inc., Yokohama, 2018.

\bibitem{ref4}
\textit{Y.~Abe},
Meromorphic function fields closed by partial derivatives, Acta Sci. Math. (Szeged) \textbf{85} (2019), 249--270.

\bibitem{ref5}
\textit{Y.~Abe},
Degenerate abelian function fields, preprint 2019, arxive: 1905.07872.

\bibitem{ref6}
\textit{J.-P.~Serre}, 
Groupes alg\'ebriques et corps de classes, Hermann, Paris, 1959.

%\bibitem{proceedings}
%\textit{P. Brooksbank},
%Article in a conference proceedings,
%in: Finite geometries (Pingree Park 2004), Oxford University Press, Oxford (2006), 1--16.

%\bibitem{phdthesis}
%\textit{J.~D.~King},
%Unpublished dissertation,
%Ph.D. thesis, University of Cambridge, 1995.


\end{thebibliography}
\end{document}